\documentclass[final,3p, times, 12pt]{elsarticle}
\usepackage[centertags]{amsmath}
\usepackage{amsfonts}
\usepackage{amssymb}
\usepackage{amsthm}

\usepackage{graphicx}
\usepackage{multicol}

\theoremstyle{plain}
\newtheorem{thm}{Theorem}[section]
\newtheorem{cor}[thm]{Corollary}
\newtheorem{lem}[thm]{Lemma}
\newtheorem{prop}[thm]{Proposition}

\theoremstyle{definition}

\newtheorem{exam}[thm]{Example}

\theoremstyle{remark}
\numberwithin{equation}{section}

\begin{document}

\begin{frontmatter}

\title{An extending result on\\
spectral radius of bipartite graphs}
\author[NCTU]{Yen-Jen Cheng\corref{cor}}\ead{yjc7755.am01g@nctu.edu.tw}
\author[HIT]{Feng-lei Fan}\ead{fanfenglei@hit.edu.cn}
\author[NCTU]{Chih-wen Weng}\ead{weng@math.nctu.edu.tw}

\cortext[cor]{Corresponding author}

\address[HIT]{Department of Photonics, Harbin Institute of Technology, 92 Xidazhi Street, Harbin, China}
\address[NCTU]{Department of Applied Mathematics, National Chiao Tung University, 1001 Ta Hsueh Road, Hsinchu, Taiwan.}

\begin{abstract}
Let $G$ denote a bipartite graph with $e$ edges without isolated vertices.
It was known  that the spectral radius  of $G$ is at most the square root of $e$,
and the upper bound is attained if and only if $G$ is a complete bipartite graph.
Suppose that $G$ is not a complete bipartite graph, and $e-1$ and $e+1$ are not twin  primes.
We determine the maximal spectral radius of $G$.
As a byproduct of our study, we obtain a spectral characterization of a pair $(e-1, e+1)$ of integers
to be a pair of twin primes.
\end{abstract}

\bigskip

\end{frontmatter}

\section{Introduction}

Let $G$ denote a bipartite graph with $e$ edges without isolated vertices.
The {\it spectral radius} of $G$ is the largest eigenvalue of the adjacency matrix of G.
It was shown in \cite[Proposition 2.1]{Bha} that the spectral radius $\rho(G)$ of $G$ satisfies
$\rho(G)\leq \sqrt{e},$
with equality if and only if $G$ is a complete bipartite graph.
There are several extending results of the above result, which aim to solve an analog of the Brualdi-Hoffman conjecture for nonbipartite graphs \cite{Hoff},  proposed in \cite{Bha}.
These extending results are scattered in \cite{Bha, cfksw10, Liu}.
To provide another extending result, we need some notations.
For $2\leq s\leq t$, let $K^-_{s, t}$ denote the graph obtained from the complete bipartite graph $K_{s, t}$ of bipartition orders $s$ and $t$ by deleting an edge, and $K_{s, t}^+$ denote the graph obtained
from $K_{s, t}$  by adding a new edge $xy$, where $x$ is a new vertex and $y$ is a vertex in the part of order $s$. Note that $K_{2, t+1}^-=K^+_{2, t},$ and $K^-_{s, t}$ and $K^+_{s, t}$ are not complete bipartite graphs.
For $e\geq 2$, let $\rho(e)$ denote the maximal value $\rho(G)$ of a bipartite graph $G$ with $e$ edges which is not a union of a complete bipartite graph and some isolated vertices if any.
For the case $(e-1, e+1)$ is not a pair of {\it twin primes}, i.e., a pair of primes with difference two,
we will determine the bipartite graph $G$ with $e$ edges such that $\rho(G)=\rho(e)$. Indeed we will show in Theorem~\ref{main1} that if $e\geq 4$ and $\rho(G)=\rho(e)$ then $G\in \{K^-_{s', t'}, K^+_{s'', t''}\},$ where
$s'$ and $t'$ (resp. $s''$ and $t''$) are chosen to minimize $s$ subject to $2\leq s\leq t$ and $e=st-1$ (resp. $e=st+1$).
The case that $(e-1, e+1)$ is a pair of twin primes is not completely solved, nevertheless we find that the shape of $\rho(e)$ in this case is lower than usual, and indeed this property characterizes a pair of  twin primes.
See Theorem~\ref{main2} for the detailed description.

\section{Preliminaries}

Let $D=(d_1, d_2, \ldots, d_p)$ be a sequence of nonincreasing positive integers of length $p$.
Let $G_D$ denote the bipartite graph with bipartition $X\cup Y$, where
$X=\{x_1, x_2, \ldots, x_p\}$ and $Y=\{y_1, y_2, \ldots, y_q\}$ ($q=d_1$),
and $x_iy_j$ is an edge if and only if $j\leq d_i.$
Note that $D$ is the degree sequence of the part $X$ in the bipartition $X\cup Y$ of $G_D.$
As $e=d_1+d_2+\cdots+d_p,$ $D$ is a {\it partition} of the number $e$ of edges in $G_D.$
The degree sequence $D^*=(d^*_1, d^*_2, \ldots, d^*_q)$ of the other part $Y$ forms the {\it conjugate partition} of $e$ as  $e=d^*_1+d^*_2+\cdots+d^*_q$, and $d^*_j=|\{i~|~d_i\geq j\}|.$  See \cite[Section 8.3]{b12} for details. The sequence $D$ will define a {\it Ferrers diagram} of $1$'s that has $p$ rows with $d_i$ $1$'s in row $i$ for $1\leq i\leq p$. For example the Ferrers diagram $F(D)$ of the sequence
 $D=(4, 2, 2, 1, 1)$ is in Figure~1. One can check that $D^*=(5, 3, 1, 1)$ in the above example.
 \bigskip

$$F(D)= \begin{tabular}{cccc}
   1 & 1 & 1 & 1 \\
   1 & 1 &   &  \\
   1 & 1 &  &   \\
   1 &  &  &  \\
   1 &  &  &  \\
 \end{tabular}$$
 \smallskip

\begin{center}
{\bf Figure 1.} The Ferrers diagram $F(D)$ of $D=(4, 2, 2, 1, 1).$
\end{center}
\bigskip

 The graph $G_D$ is important in the study of the spectral radius
of bipartite graphs with prescribed degree sequence of one part of the bipartition.

\begin{lem}(\cite[Theorem 3.1]{Bha})\label{lemD}
 Let $G$ be a bipartite graph without isolated vertices such that one part in the bipartition of $G$ has degree sequence $D=(d_1,\ldots,d_p)$. Then $\rho(G)\le\rho(G_D)$ with  equality if and only if $G=G_D$ (up to isomorphism).
\end{lem}

The idea of the proof of Lemma~\ref{lemD} may be traced back to \cite{s64}.
Let $(u_1,u_2,\ldots,u_p;v_1,v_2,\ldots,v_q)$ be a positive Perron eigenvector of the adjacency matrix of $G$, where vertices in the part $Y$ of bipartition $X\cup Y$ of $G$ are ordered to ensure $v_1\geq v_2\geq\cdots\geq v_q$, i.e., the latter part of the positive Perron eigenvector nonincreasing. For $1\leq i<j\leq q,$ if $x_ky_{j}$ is an edge and $x_ky_i$ is not an edge in $G$ for some $x_k\in X$, then the new bipartite graph $G'$ with the same vertex set as $G$ obtained by deleting the edge $x_ky_{j}$ and adding a new edge $x_ky_i$ has spectral radius $\rho(G')\geq \rho(G)$.

A bipartite graph $G$ is {\it biregular} if the degrees of vertices in the same part of its bipartition are the same constant. Let $H$, $H'$ be two bipartite graphs with given ordered bipartitions $VH = X\bigcup Y$ and $VH' = X'\bigcup Y'$, where $VH\bigcap VH' = \phi$. The {\it bipartite sum} $H + H'$ of $H$ and
$H'$ (with respect to the given ordered bipartitions) is the graph obtained from $H$ and
$H'$ by adding an edge between $x$ and $y$ for each pair $(x,y) \in X \times Y'\bigcup   X' \times Y$.
Chia-an Liu and the third author  \cite{Liu} found upper bounds of $\rho(G)$ expressed by degree sequences of two parts of the bipartition of $G$.

\begin{lem}(\cite{Liu})
\label{lemphi}
Let $G$ be a bipartite graph with bipartition $X\cup Y$ of orders $p$ and $q$ respectively such that
the part $X$ has degree sequence $D=(d_1,\ldots,d_p)$, and the other part $Y$ has degree sequence $D'=(d'_1,d'_2\ldots, d'_q)$, both in nonincreasing order.
For $1\le s\le p$ and $1\le t\le q$, let $X_{s,t}=d_sd'_t+\sum_{i=1}^{s-1}{(d_i-d_s)}+\sum_{j=1}^{t-1}{(d'_j-d'_t)}$, $Y_{s,t}=\sum_{i=1}^{s-1}{(d_i-d_s)}\cdot\sum_{j=1}^{t-1}{(d'_j-d'_t)}$.
Then
$$\rho(G)\le\phi_{s,t}:=\sqrt{\frac{X_{s,t}+\sqrt{X_{s,t}^2-4Y_{s,t}}}{2}}.$$
Furthermore, if $G$ is connected then the above equality holds if and only if there exist
nonnegative integers $s' < s$ and $t' < t$, and a biregular graph $H$ of bipartition orders
$p-s'$ and $q-t'$ respectively such that $G = K_{s',t'} + H$.
\end{lem}

The idea of the proof in Lemma~\ref{lemphi} is to apply Perron-Frobenius Theorem for spectral radius to matrices that are similar to the adjacency matrix of $G$ by diagonal matrices with variables on diagonals. Results using this powerful method are also in \cite{cls13a,ctg13,hy14,hw14,lw13,sw04,xz13,xz14}.

\section{Graphs closed to $K_{p, q}$}

Applying Lemma~\ref{lemphi} to the graph $G=G_D$
 for a given sequence $D=(d_1, d_2, \ldots, d_p)$ of nonincreasing positive integers of length $p$, one immediately finds that $d'_j=d^*_j$ and
$$\sum_{j=1}^{t-1}{(d'_j-d'_t)}=\sum_{i=d'_{t}+1}^pd_i.$$
Moreover if $s$ is chosen such that $d_s<d_{s-1}$ and $t=d_{s}+1$, then $d'_t=s-1$ and the corresponding
Ferrers diagram $F(D)$ has a blank in the $(s,t)$ position, so
$$X_{s,t}=d_s(s-1)+\sum_{i=1}^{s-1}{(d_i-d_s)}+\sum_{i=s}^pd_i=e$$
and
\begin{equation}\label{eY}Y_{s,t}=\sum_{i=1}^{s-1}{(d_i-d_s)}\cdot \sum_{i=s}^pd_i,
\end{equation}
completely expressed by $D.$ Hence we have the following simpler form of Lemma~\ref{lemphi}.

\begin{lem}\label{pr1}
Assume that $s$ is chosen satisfying $d_s<d_{s-1}$ in the sequence $D=(d_1, d_2, \ldots, d_p)$ of positive integers and $e=d_1+d_2+\cdots+d_p$.
Then $$\rho(G_D)\le \sqrt{\frac{e+\sqrt{e^2-4 \sum_{i=1}^{s-1}{(d_i-d_s)}\cdot \sum_{i=s}^pd_i }}{2}},$$
with equality if and only if $D$ contains exactly two different values. \qed
\end{lem}

The following are a few special cases.

\begin{exam}(\cite{Liu})\label{examp} Suppose that $2\leq p\leq q$ and $K_{p, q}^e$ (resp. ${^e}K_{p, q}$) is the graph
obtained from $K_{p, q}$ by deleting $k:=pq-e$ edges incident on a common vertex in the part of order $q$ (resp. $p$). Then
\begin{align*}
\rho(K_{p, q}^e)&=\sqrt{\frac{e+\sqrt{e^2-4k(q-1)(p-k)}}{2}}\qquad (k=pq-e<p),\\
\rho({^e}K_{p,q})&=\sqrt{\frac{e+\sqrt{e^2-4k(p-1)(q-k)}}{2}}\qquad (k=pq-e<q).
\end{align*}\qed
\end{exam}

Applying Example~\ref{examp} to the graph
$K^-_{p, q}=K^{pq-1}_{p, q}={^{pq-1}}K_{p, q}$, one immediate finds that
$$\rho(K^-_{p, q})= \sqrt{\frac{e+\sqrt{e^2-4(e-(p+q)+2)}}{2}},$$
which obtains maximum (resp. minimum) when $p$ is minimum (resp. $p$ is maximum) subject to the fixed number $e=pq-1$ of edges and $2\leq p\leq q$. Note that
$$e-(p+q)+2\leq e-2\sqrt{pq}+2=e-2\sqrt{e+1}+2<e-1-\sqrt{e-1}\qquad (e\geq 6).$$
Hence
$$\rho(K^-_{p, q})>\sqrt{\frac{e+\sqrt{e^2-4(e-1-\sqrt{e-1})}}{2}}\qquad (q\geq p\geq 3).$$
As $K^-_{2, 2}$ has $3$ edges, one can check that \begin{equation}\label{e=3}\rho (K^-_{2,2})=\sqrt{\frac{3+\sqrt{5}}{2}}<\sqrt{\frac{e+\sqrt{e^2-4(e-1-\sqrt{e-1})}}{2}}.\end{equation}
 Similarly $K^+_{p, q}=K^{pq+1}_{p, q+1}$ has spectral radius
\begin{equation}\label{eq+}\rho(K^+_{p, q})=
                       \sqrt{\frac{e+\sqrt{e^2-4(e-1-q)}}{2}},\end{equation}
which obtains maximum (resp. minimum) when $p$ is minimum (resp. $p$ is maximum) subject to the fixed number $e=pq+1$ and $2\leq p\leq q$. Note that $e-1-q\leq e-1-\sqrt{e-1}$ in this case. This proves the following lemma.
\begin{lem}\label{comp} The following (i)-(iii) hold.
\begin{enumerate}
\item[(i)] For all positive integers $2\leq p'\leq q',(p',q')\ne (2,2) $, $2\leq p''\leq q''$ satisfying $e=p'q'-1=p''q''+1,$ we have
$$\rho(K^-_{p', q'}), \rho(K^+_{p'', q''})\geq \sqrt{\frac{e+\sqrt{e^2-4(e-1-\sqrt{e-1})}}{2}}.$$ Moreover
 the above equality does not hold for $\rho(K^-_{p',q'})$, and holds for $\rho(K^+_{p'',q''})$ if and only if $p''=q''$.
\item[(ii)] If $e+1$ is not a prime and $p'\geq 2$ is the least integer such that $p'$ divides $e+1$ and $q':=(e+1)/p',$ then for any positive integers $2\leq p\leq q$ with $e=pq-1$, we have
$\rho(K^-_{p, q})\leq \rho(K^-_{p', q'})$, with equality if and only if  $(p, q)=(p', q').$
\item[(iii)] If $e-1$ is not a prime, and $p''\geq 2$ is the least integer such that $p''$ divides $e-1$ and $q'':=(e-1)/p''$, then for positive integers $2\leq p\leq q$ with $e=pq+1$, we have
$\rho(K^+_{p, q})\leq \rho(K^+_{p'', q''})$, with equality if and only if $(p, q)=(p'', q'')$.
\end{enumerate}
\qed
\end{lem}

Note that the condition $2\leq p'\leq q',(p',q')\ne (2,2)$ in (i) is from the previous condition $3\leq p'\leq q'$ and $K^-_{2,q}=K^+_{2,q-1}$ for $q\geq 3$.

\section{Graphs with at least two edges different from $K_{p, q}$}

In this section, we consider bipartite graphs which are not complete bipartite and are not considered in Lemma~\ref{comp}. The following lemma is for the special case that the graph has the form $G=G_D$.

\begin{lem}\label{lem1} Let $D=(d_1, d_2, \ldots, d_p)$ be a partition of $e.$
Suppose that $G_D$ is not a complete bipartite graph and is not one of the graphs $K^-_{p', q'}$ or $K^+_{p'', q''}$ for any $2\leq p'\leq q',(p',q')\ne (2,2),$ $2\leq p''\leq q''$ such that $e=p'q'-1=p''q''+1.$ Then
$$\rho(G_D)< \sqrt{\frac{e+\sqrt{e^2-4(e-1-\sqrt{e-1})}}{2}}.$$
\end{lem}

\begin{proof} When $e\leq 3$, $G_D=K^-_{2,2}$ is the only graph and the inequality holds by (\ref{e=3}). We assume that $e\geq 4$. The assumption implies that $p, q:=d_1\geq 2$ and $4\leq e\leq pq-2$.
Using $D^*$ to replace $D$ if necessary, we might assume that $2\leq p\leq q$ and $q\geq 3.$
Since $G_D$ is not complete, we choose $s$ such that $1\leq s\leq p $ and $d_{s-1}>d_s.$ Set $t=d_s+1$.
According to the partition $(s-1, 1, p-s)$ of rows and the partition $(t-1, 1, q-t)$ of columns,
 the Ferrers diagram $F(D)$ is divided into $9$ blocks and the number $b_{ij}$ of $1$'s in the block $(i, j)$ for $1\leq i, j\leq 3$ is shown as
 $$\begin{bmatrix}
b_{11}& b_{12} & b_{13}\\
b_{21}     & b_{22} & b_{23} \\
b_{31}& b_{32} & b_{33}
\end{bmatrix}=\begin{bmatrix}
(s-1)d_s& s-1 & \sum_{i=1}^{s-1}{(d_i-d_s-1)}\\
d_s     & 0 & 0 \\
\sum_{i=s+1}^pd_i& 0 & 0
\end{bmatrix}.$$
Note that $b_{11}=b_{12}b_{21}$ and  $b_{11}+b_{12}+b_{13}+b_{21}+b_{31}=e.$
Referring to Lemma~\ref{pr1} and (\ref{eY}), it suffices to show that $Y_{s, t}>e-1-\sqrt{e-1}.$
Note that
\begin{align*}Y_{s, t}=&\sum_{i=1}^{s-1}{(d_i-d_s)}\cdot \sum_{i=s}^pd_i=(s-1+\sum_{i=1}^{s-1}{(d_i-d_s-1)}) (d_s+ \sum_{i=s+1}^pd_i )\\
=&(b_{12}+b_{13})(b_{21}+b_{31})=b_{11}+b_{12}b_{31}+b_{21}b_{13}+b_{13}b_{31}.
\end{align*}
Note that $b_{12}b_{21}\not=0$, and that $G\not=K^-_{p', q'}$ implies that $b_{13}\not=0$ or $b_{31}\not=0.$ If both parts $b_{13}$ and $b_{31}$ are not zero then
$b_{12}b_{31}\geq b_{12}+b_{31}-1,$ $b_{21}b_{13}\geq b_{21}+b_{13}-1,$ and  $b_{13}b_{31}\geq 1,$
so $Y_{s, t}\geq b_{11}+(b_{12}+b_{31}-1)+(b_{21}+b_{13}-1)+1=e-1>e-1-\sqrt{e-1}.$ The proof is completed.
The above proof holds for any $s$ with $d_{s-1}<d_s.$ We choose the least one with such property, and
might assume one of the following two cases (i)-(ii).
\smallskip

{\bf Case} (i).  $b_{31}=0$ and $b_{13}\not=0$: Then $s=p=b_{12}+1\geq 2,$
and $G={^e}K_{p, q}$, where $e=pq-(q-d_p)\geq p^2-p+1>(p-1/2)^2+3/4.$
 Thus
$$Y_{s, t}= b_{11}+b_{21}b_{13}\geq e-1-b_{12}=e-p > e-1- \sqrt{e-1}.$$

{\bf Case} (ii).  $b_{13}=0$ and $b_{31}\not=0$: The condition $b_{13}=0$ implies that $t=q$ and $b_{21}=q-1\geq 2.$ The proof is further divided into the following two cases (iia) and (iib).
\smallskip

{\bf Case} (iia). $1\leq b_{31}<b_{21}$: If $s<p-1$, let $s'=s+1$ and $t'=d_{s'}+1$. Then $d_{s'-1}>d_{s'}$ and $d_{s'+1}\not=0$.  Let $b'_{ij}$ be the $b_{ij}$ corresponding to the new choice of $s'$ and $t'$.
Then $b'_{13}b'_{31}\not=0$ and the proof is completed as in the beginning.
Note that $s\not=p$ since $b_{31}\not=0.$ Then we may assume $s=p-1.$
This implies that $b_{31}=d_p<q-1$ and $e = pq-1-q+d_p\geq p^2-p=(p-1/2)^2-1/4.$
Let $s'=p$ and $t'=d_p+1$, and then
$$Y_{s', t'}=b'_{21}(b'_{12}+b'_{13})\geq e-1-b'_{12}=e-p> e-1- \sqrt{e-1}.$$

{\bf Case} (iib).  $b_{31}\geq b_{21}$: If $b_{12}=1$ then by the assumption $G\not=K^+_{p'', q''},$
there exists another $s''>s$ such that $d_{s''}<d_{s''-1}$. Apply the above proof on $(s, t)=(s'', t'')$.
Since $b''_{13}\geq 1$, we might assume $b''_{31}=0$. Then $s''=p$ and
$e=(p-1)(q-1)+d_p+1\geq (p-1)^2+2.$ Hence
$$Y_{s'', t''}=b''_{21}(b''_{12}+b''_{13})\geq e-1-b''_{12}=e-p> e-1- \sqrt{e-1}.$$
We now assume in the last situation that $b_{12}>1.$ Then
$$Y_{s, t}=b_{11}+(b_{12}-1)b_{31}+b_{31}\geq b_{11}+b_{12}+2b_{31}-2\geq e-2> e-1- \sqrt{e-1}.$$
\end{proof}

We now study the general case.

\begin{prop}\label{bound_e}
Let $G$ be a bipartite graph without isolated vertices which is not a complete bipartite graph and one of the graphs $K^-_{p', q'},$ $K^+_{p'', q''}$ for any $2\leq p'\leq q',(p',q')\ne (2,2),$ $2\leq p''\leq q'',$ such that
$e=p'q'-1=p''q''+1$ is the number of edges in $G$. Then
$$\rho(G)< \sqrt{\frac{e+\sqrt{e^2-4(e-1-\sqrt{e-1})}}{2}}.$$
\end{prop}

\begin{proof}
Let $G_D$ be the graph obtained from a degree sequence $D$ of any part, say $X$, in the bipartition $X\cup Y$ of $G$.
Then $\rho(G)\leq \rho(G_D)$ by Lemma~\ref{lemD}. The proof is finished if $G_D$ satisfies
the assumption of Lemma~\ref{lem1}.
Let $D'$ be the degree sequence of the other part $Y$ in the bipartition of $G$.
Then we might assume that $G\not=G_D$, $G\not=G_{D'}$, and $G_D$ and $G_{D'}$ are graphs of the forms
$K_{p, q}$, $K^-_{p', q'},$ or $K^+_{p'', q''}$.
For $y_i\in Y$, let $N(y_i)$ be the set of neighbors of $y_i$ in $G$.
Suppose for this moment that $|N(y_i)|=|N(y_j)|$ and $N(y_i)\not=N(y_j)$ for some $y_i, y_j\in Y$.
Assume that $y_i$ is before $y_j$ in the order that makes  the entries in
the latter part of the positive Perron eigenvector described after Lemma~\ref{lemD} nonincreasing.
Let $G''$ be the bipartite graph obtained from $G$ by moving an edge incident on $y_j$ but not on $y_i$
to incident on $y_i$, keeping the other endpoint of this edge unchanged.  Let $D''$ be the new degree sequence on the part $Y$ of the new bipartite graph $G''.$ Then $\rho(G)\leq \rho(G'')\leq \rho(G_{D''})$.
Noticing that $D''$ is obtained from $D'$  by replacing two given equal values $a$ by $a-1$ and $a+1$. Hence $G_{D''}$ is not a graph of the form $K_{p, q}$, $K^-_{p', q'},$ or $K^+_{p'', q''}$. Thus the proof follows from Lemma~\ref{lem1}. Hence we might assume that if $|N(y_i)|=|N(y_j)|$ then $N(y_i)=N(y_j)$ for all $y_i, y_j\in Y.$  Note that $D$ has at most two distinct values, and so does $D'$. Reordering the vertices in $Y$ such that the former has larger degree and then doing the same thing for $X$, we find indeed $G=G_D=G_{D'}$, a contradiction.
\end{proof}

We provide two applications of Proposition~\ref{bound_e}.

\begin{cor}\label{twin} Let $G$ be a bipartite graph with $e$ edges without isolated vertices. Suppose that $G$ is not a complete bipartite graph, and $(e-1, e+1)$ is a pair of twin primes.
Then
$$\rho(G)< \sqrt{\frac{e+\sqrt{e^2-4(e-1-\sqrt{e-1})}}{2}}.$$
\end{cor}
\begin{proof}
If $(e-1, e+1)$ is a pair of primes then there is no way to express $G$ as a graph of the forms
$K^-_{p', q'}$ or $K^+_{p'', q''}$. The proof follows from Proposition~\ref{bound_e}.
\end{proof}

\begin{cor}\label{old}
Let $G$ be a bipartite graph without isolated vertices which is not one of the graphs $K_{p, q}$, $K^-_{p', q'},$ $K^+_{p'', q''}$ for any $1\leq p\leq q,$ $2\leq p'\leq q',$ $2\leq p''\leq q''$ such that
$e=pq=p'q'-1=p''q''+1$ is the number of edges in $G$.
Assume that $e=st+1$ (resp. $e=st-1$) for $2\leq s \leq t$. Then
$$\rho(G)< \rho(K_{s, t}^+)\qquad ({\rm resp.}~\rho(G)< \rho(K_{s, t}^-)).$$
\end{cor}
\begin{proof} If $s=t=2$ and $e=st-1=3$ then either $G=3K_2$ the disjoint union of three edges or $G=K_{1, 2}\cup K_2$ the disjoint of a path of order $3$ and an edge. One can easily check that $\rho(G)< \rho(K_{2, 2}^-).$ The remaining cases are  from Proposition~\ref{bound_e} and Lemma~\ref{comp}(i) and noticing that
 $K_{2, t+1}^-=K^+_{2, t}$ for $t\geq 2$.
\end{proof}

It is worth mentioning that the result $\rho(G)< \rho(K_{s, t}^-)$ in Corollary~\ref{old} had also been proven in \cite[Theorem 8.1]{Bha} under more assumptions.

\section{Main Theorems}

For $e\geq 2$, recall that $\rho (e)$ is
the maximal value $\rho(G)$ of a bipartite graph $G$ with $e$ edges which is not a union of a complete bipartite graph and some isolated vertices if any.
 Note that $$\rho(2)=\rho(2K_2)=1,~ {\rm and~~} \rho(3)=\rho(K^-_{2,2})=\sqrt{\frac{3+\sqrt{5}}{2}}.$$
Two theorems about $\rho(e)$  are given in this section.

\begin{thm}\label{main1}
Let $G$ be a bipartite graph with $e\geq 4$ edges without isolated vertices such that $\rho(G)=\rho(e)$.
Then the following (i)--(iv) hold.
\begin{enumerate}
\item[(i)] If $e$ is odd then $G=K_{2, q}^-$, where $q=(e+1)/2.$
\item[(ii)] If $e$ is even, $e-1$ is a prime and $e+1$ is not a prime, then $G=K_{p', q'}^-,$ where $p'\geq 3$ is the least integer that divides $e+1$ and $q'=(e+1)/p'.$
\item[(iii)] If $e$ is even, $e-1$ is not a prime and $e+1$ is a prime, then $G=K_{p'', q''}^+,$ where $p''\geq 3$ is the least integer that divides $e-1$ and $q''=(e-1)/p''.$
\item[(iv)] If $e$ is even and neither $e-1$ nor $e+1$ is a prime, then $G\in \{K_{p', q'}^-, K_{p'', q''}^+\}$, where $p', q'$ are as in (ii) and $p'', q''$ are as in (iii).
\end{enumerate}
\end{thm}
\begin{proof}
By the definition of $\rho(e)$ and the fact \cite[Proposition 2.1]{Bha} which is mentioned in the introduction, $G$ is not a complete graph.
 From Lemma~\ref{comp}(i) and Proposition~\ref{bound_e}, we only need to
compare the spectral radii  $\rho(K^-_{p, q})$ and $\rho(K^+_{p, q})$  for all possible positive integers $2\leq p\leq q$ that keep the graphs have $e$ edges.
This has been done in Lemma~\ref{comp}(ii)-(iii).
\end{proof}

\begin{thm}\label{main2} Let $e\geq 4$ be an integer. Then $(e-1, e+1)$ is a pair of twin primes if and only if
$$\rho(e)<\sqrt{\frac{e+\sqrt{e^2-4(e-1-\sqrt{e-1})}}{2}}.$$
\end{thm}
\begin{proof}
The necessity is by Corollary~\ref{twin}. The sufficiency is from Theorem~\ref{main1} and Lemma~\ref{comp}(i).
\end{proof}

Due to Yitang Zhang's recent result \cite{z14}, the conjecture if there are infinite pairs of twin primes obtains much attention. Theorem~\ref{main2} provides a spectral description of the pairs of twin primes.

\section{Numerical comparisons}

In the case (iv) of Theorem~\ref{main1}, the two graphs $K_{p', q'}^-$ and $K_{p'', q''}^+ $ are candidates to be extremal graph. For even $e\leq 100$ and neither $e-1$ nor $e+1$ is a prime, we shall determine
which graph has larger spectral radius.
The symbol $-$ in the last column of the following table means that $K^-_{p',q'}$ wins, i.e. $\rho(K_{p', q'}^-)>\rho(K_{p'', q''}^+)$
and $+$ otherwise.
\bigskip

\begin{center}
\begin{tabular}{|c|c|c|c|}
  \hline
    $e$ & $\rho(K_{p', q'}^-)$ & $\rho(K_{p'', q''}^+)$ & winner \\  \hline
    $26$ & $\sqrt{13+3\sqrt{17}}$ & $\sqrt{13+\sqrt{149}}$ & $-$ \\   \hline
    $34$ &  $\sqrt{17+\sqrt{265}}$ & $\sqrt{17+\sqrt{267}}$  & $+$   \\   \hline
    $50$ & $\sqrt{25+\sqrt{593}}$   &  $\sqrt{25+\sqrt{583}}$ & $-$   \\   \hline
    $56$ & $\sqrt{28+\sqrt{748}}$  & $\sqrt{28+\sqrt{740}}$  & $-$  \\   \hline
    $64$ & $\sqrt{32+\sqrt{976}}$  & $\sqrt{32+\sqrt{982}}$  &  $+$ \\   \hline
    $76$ & $\sqrt{38+\sqrt{1384}}$  & $\sqrt{38+\sqrt{1394}}$  & $+$  \\ \hline
    $86$ & $\sqrt{43+\sqrt{1813}}$  & $\sqrt{43+\sqrt{1781}}$  & $-$  \\ \hline
    $92$ & $\sqrt{46+\sqrt{2096}}$  & $\sqrt{46+\sqrt{2078}}$  & $-$  \\ \hline
    $94$ & $\sqrt{47+\sqrt{2137}}$  & $\sqrt{47+\sqrt{2147}}$  & $+$  \\
  \hline
\end{tabular}

\bigskip

{\bf Table. Comparisons of $\rho(K_{p', q'}^-)$ and $\rho(K_{p'', q''}^+)$ in case (iv) of Theorem~\ref{main1} for $e\leq 100$}

\end{center}

\section*{Acknowledgments}
The second author is honored to be an exchange student from HIT to NCTU and thank Dr. Chia-an Liu for his introduction of Lemma 2.2 which turns out to be an important tool of this paper. This research is supported by the
Ministry of Science and Technology of Taiwan R.O.C. under the project NSC 102-2115-M-009-009-MY3.

\end{document}